\documentclass[10pt,a4paper]{amsart}
\usepackage[left=3.5cm,right=3.5cm,top=3.5cm,bottom=3.5cm]{geometry}
\usepackage{enumerate}
\usepackage{leftidx}
\usepackage{mathtools}
\usepackage{amsmath,amscd,amssymb,amsthm}
\usepackage{bm}
\usepackage{comment}
\usepackage[arrow,matrix]{xy}
\usepackage[OT2,T1]{fontenc}
\usepackage{booktabs}
\usepackage[position=bottom]{subfig}
\usepackage{xfrac, faktor}
\usepackage{lipsum}
\usepackage{graphicx}
\usepackage{setspace,color}
\usepackage{hyperref}
\numberwithin{equation}{section}
\numberwithin{table}{section}
\usepackage{array}
\usepackage{multirow}
\usepackage{algorithm}
\usepackage{algpseudocode}
\allowdisplaybreaks

\DeclareMathOperator{\Aut}{Aut}

\DeclareMathOperator{\rank}{rank}
\DeclareMathOperator{\Mat}{Mat}

\DeclareMathOperator{\End}{End}

\DeclareMathOperator{\ev}{ev}

\DeclareMathOperator{\Nr}{Nr}
\DeclareMathOperator{\ord}{ord}
\DeclareMathOperator{\Hom}{Hom}
\DeclareMathOperator{\Cent}{Cent}

\theoremstyle{definition}
\newtheorem{definition}{Definition}[section]
\newtheorem{example}[definition]{Example}
\newtheorem{remark}[definition]{Remark}
\newtheorem*{remark*}{Remark}

\theoremstyle{plain}
\newtheorem{theorem}[definition]{Theorem}
\newtheorem{corollary}[definition]{Corollary}
\newtheorem{lemma}[definition]{Lemma}
\newtheorem{proposition}[definition]{Proposition}

\newcommand{\vF}{{ \mathbb F }}
\newcommand{\N}{{ \mathbb N }}
\newcommand{\vN}{{ \mathbb N }}

\newcommand{\vR}{{ \mathcal R }}

\renewcommand{\epsilon}{\varepsilon}

\title{On the Characteristic Polynomial of Linearized Polynomials}

\author[L. Bastioni]{Luca Bastioni}
\address{University of South Florida \\ 4202 E Fowler Ave\\
33620 Tampa, US.}
\email{lbastioni@usf.edu}

\author[G. Micheli]{Giacomo Micheli}
\address{University of South Florida \\ 4202 E Fowler Ave\\
33620 Tampa, US.}
\email{gmicheli@usf.edu}

\author[S. Zhao]{Shujun Zhao}
\address{University of South Florida \\ 4202 E Fowler Ave\\
33620 Tampa, US.}
\email{shujunz@usf.edu}


\makeatletter
\@namedef{subjclassname@2020}{%
  \textup{2020} Mathematics Subject Classification}
\makeatother
\begin{document}

\begin{abstract}
Let $k$ be a finite field, and $L$ be a $q$-linearized polynomial defined over $k$ of $q$-degree $r$ ($L=\sum^r_{i=0}a_iZ^{q^i}$, with $a_i\in k$).
This paper provides an algorithm to compute a characteristic polynomial of $L$  over a large extension field $\vF_{q^n}\supseteq k$. Our algorithm has computational complexity of $O(n(\log(n))^4)$ in terms of $\vF_q$ operations with the implied constant depending only on $k$ and $r$. Up to logarithmic factors, and for linear maps represented by low degree polynomials, this provides a square root improvement over generic algorithms.
\end{abstract}

\maketitle

\section{Introduction}

\subsection{Characteristic Polynomial of a Matrix}
Let $A\in \Mat_n(K)$, where $K$ is a field. Computing the characteristic polynomial of $A$ is a classical problem. In 1985, Keller-Gehrig \cite{keller1985fast} reduced this problem to a matrix multiplication problem, obtaining an algorithm with complexity $O(n^\omega\log(n))$ in general, and $O(n^\omega)$ when $A$ is generic matrix: each of whose coefficients can be considered as an independent indeterminate. Here, $\omega$ denotes the exponent of the optimal time complexity of matrix multiplication, for which it's known that $2\leq \omega \leq 3$, with the current best bound being $\omega < 2.371866$ \cite{duan2023faster}. In 2007, C. Pernet and A. Storjohann \cite{pernet2007faster} showed that the characteristic polynomial can be computed with expected cost $O(n^\omega)$, improving by a factor of $\log(n)$ Keller-Gehrig's bound, with a success probability of at least 1/2, provided the field $F$ contains at least $2n^2$ elements. More recently, in 2021, V. Neiger and C. Pernet \cite{neiger2021deterministic} removed both the randomness and field size restriction, presenting a deterministic algorithm that achieves complexity $O(n^\omega)$ for computing the characteristic polynomial of any matrix over an arbitrary field $K$.

\subsection{Characteristic Polynomial of Endomorphisms of Drinfeld Modules}
Let $A=\vF_q[T]$ be the ring of polynomials. Let $\phi$ be a Drinfeld module of rank $r$ over $k=\vF_{q^n}$ with $A$-characteristic $\mathfrak{p}$ of degree $d$ (See subsection \ref{subsec:drinfeldintro}). In \cite{musleh2023computing}, the authors Y. Musleh and É. Schost developed an algorithm to compute the characteristic polynomial of an arbitrary endomorphism $u\in\End_k(\phi)$ of a finite Drinfeld module using its associated crystalline cohomology. The algorithm is faster when $u=\tau^n$ is the Frobenius endomorphism and $d=n$. In this case, the algorithm attains a bit complexity of
\begin{equation}\label{eq:primecase}
(r^{\omega} n^{1.5} \log q + n \log^2 q)^{1 + o(1)},
\end{equation}
where $\omega$ is the exponent of the optimal time complexity of matrix multiplication as mentioned before. For $u=\tau^n$ and $d<n$, the characteristic polynomial could be computed with bit complexity 
\begin{equation}\label{eq:generalcase}
(r^{\lambda}/d + r^{\omega}/\sqrt{d}) n^2 \log q + n \log^2 q)^{1 + o(1)},
\end{equation}
where $\lambda$ denotes an exponent such that the characteristic polynomial of an $s\times s$ matrix over a ring $R$ can be computed in $O(s^\lambda)$. When $R$ is a field this can be done with the same cost of matrix multiplication and so $\lambda=\omega$. For general rings, the best known value as of today is $\lambda\approx 2.7$ (see \cite{kaltofen2005complexity}). When $r$ and $q$ is fixed, \ref{eq:primecase} is essentially linear in $n^{1.5}$ and \ref{eq:generalcase} is linear in $n^2$.

Another important contribution is due to X. Caruso and A. Leudière \cite{caruso2025algorithms}, who proposed several algorithms to compute characteristic polynomials of endomorphisms and norms of isogenies of Drinfeld modules. Among these, their most efficient algorithm for computing the characteristic polynomial of the Frobenius endomorphism achieves a cost of $O((n^2r+nr^\omega)\log q)$ operations in $k$.

\subsection{A square root improvement to the computation of the characteristic polynomial of a low degree linearized polynomial}
Let $k$ be a finite field, and $L$ be a $q$-linearized polynomial defined over $k$ of $q$-degree $r$ ($L=\sum^r_{i=0}a_iZ^{q^i}$, with $a_i\in k$).
This paper provides an algorithm to compute a characteristic polynomial of $L$  over a large extension field $\vF_{q^n}\supseteq k$. Our algorithm has computational complexity of $O(n(\log(n))^4)$ in terms of $\vF_q$ operations with the implied constant depending only on $k$ and $r$. This is essentially better than a square root improvement (up to logarithmic factors) of the  state of art: as mentioned before, if one were to compute the characteristic polynomial of $L$ as an $\vF_q$ linear map over $\vF_{q^n}$ one would have a complexity of $O(n^\omega)$, where $\omega\in [2,3]$. In other words, in this paper we show that if a linear map comes from a $q$-polynomial of low degree, then its characteristic polynomial computation is much faster than the one of linear maps having a large degree $q$-polynomial representing it.

\section{Preliminaries}

\subsection{Drinfeld Modules}\label{subsec:drinfeldintro}
In this subsection, we provide some background on the theory of Drinfeld modules. For a comprehensive treatise of the subject, the reader should refer to \cite{Papikian2023Drinfeld}. Throughout the article, $q$ is a prime power and $\vF_q$ the finite field with $q$ elements. We denote by $A=\vF_q[T]$ the ring of polynomials over $\vF_q$ and by $\vF_q(T)$ the rational function field in the variable $T$. Let $k\supseteq \vF_q$ be a field and $a\in A$.

\begin{definition}
Let $n\geq 0$. A polynomial $f(Z)=\sum_{i=0}^na_iZ^{q^i}$ in $k[Z]$ is called a $q$-\emph{linearized} polynomial. If $\operatorname{deg}(f)=q^n$, we say that $f$ has $q$-degree equal to $n$ and we denote $q^i$ by $[i]$. The ring of $q$-linearized polynomials is denoted by $k\langle Z \rangle$, and in general, it is a noncommutative ring under addition and composition.      
\end{definition}

\begin{definition}
Let $k$ be a commutative $\vF_q$-algebra and $\tau$ be an indeterminate.  Let $k\{\tau\}$ be the set of polynomials of the form $\sum_{i=0}^na_i\tau^i$ such that $a_i\in k$ and $n\geq 0$. In $k\{\tau\}$, for $\{a_i,b_i\}_{i=0,\ldots,n}\in k$, we define the addition as the usual addition of polynomials
\[\sum_{i=0}^na_i\tau^i+\sum_{i=0}^nb_i\tau^i=\sum_{i=0}^n(a_i+b_i)\tau^i.
\]
The multiplication is instead defined by the following rule: for any $a,b\in k$
\begin{equation}\label{twistedmultiplication}
(a\tau^i)(b\tau^j):=ab^{q^i}\tau^{i+j}.    
\end{equation}
With these operations $k\{\tau\}$ is a non-commutative ring called the ring of \textit{twisted} polynomials. 
\end{definition}

Notice that $k\{\tau\}$ is related to the ring of $q$-linearized polynomials $k\langle Z \rangle$ via the following isomorphism: 
\begin{align*}
\iota: k\{\tau \}&\longrightarrow k\langle Z \rangle\\ 
a_i\tau^i&\longmapsto a_iZ^{q^i}.
\end{align*}
With a little abuse of notation, for $f\in k\{\tau\}$ we write $f(Z)$ instead of $\iota(f)(Z)$.

\begin{definition}\label{drinfeldmodule}
Let $k$ be an $A$-field, (i.e. a field equipped with an $\vF_q$-algebra homomorphism $\gamma:\vF_q[T]\rightarrow k$).
A Drinfeld module of rank $r\geq 1$ is an $\vF_q$-algebra homomorphism 
\begin{align*}
\phi:\vF_q[T]\longrightarrow & k\{\tau\}\\
a\longmapsto &\phi_a=\gamma(a)+\sum^n_{i=1}h_i(a)\tau^i
\end{align*}
where $h_n(a)\neq 0$ and $n=r\cdot\deg_T(a)$. The morphism $\gamma$ is called the \emph{structure morphism} of $\phi$, and $\mathfrak{p}\coloneqq\ker(\gamma)$ is called the $A$-\emph{characteristic} of $k$. Say that $d=\deg \mathfrak{p}$.
\end{definition}
Observe that for any  $a=\sum_{i=0}^na_iT^i$, we have $\phi_a=\sum_{i=0}^na_i\phi_T^i$. In other words, $\phi_T$ determines $\phi$ entirely. Also notice that via $\phi$, $k$ acquires a twisted $A$-module structure simply defining $a\circ\beta \coloneqq \phi_a(\beta)$ for every $a\in A$ and $\beta\in k$. We denote such $A$-module by $^\phi k$.

\begin{definition}
Let $\phi, \psi$ be Drinfeld modules defined over $k$. A \emph{morphism} $u: \phi \to \psi$ is a twisted polynomial $u \in k\{\tau\}$ such that $u \phi_a = \psi_a u$ for all $a \in A$. The group of all morphisms $\phi \to \psi$ over $k$ is denoted by $\Hom_k(\phi, \psi)$. We also define 
\[
\End_k(\phi)\coloneqq \Hom_k(\phi,\phi)=\Cent_{k\{\tau\}}(\phi(A)).\]
The set $\End_k(\phi)$ admits a ring structure via sum and product of morphisms, and it is called the \emph{endomophism ring} of $\phi$. When $k=\vF_{q^n}$, then $\tau^n\in\End_k(\phi)$ is called the \emph{Frobenius} endomorphism of $\phi$.
\end{definition}

The following proposition is a standard fact, we include the proof for completeness.

\begin{proposition}\label{theo:integral}
Let $\phi$ be a Drinfeld module of rank $r$ over $k$. Let $u\in\End_k(\phi)$. Then $\vF_q[T,u]$ is an integral extension of $\vF_q[T]$.
\end{proposition}
\begin{proof}
It is enough to prove that $u$ is integral over $\vF_q[T]$. To see this, observe that $\End_k(\phi)$ is a free module of finite rank over $\vF_q[T]$ (see \cite[Theorem 3.4.1]{Papikian2023Drinfeld}), and therefore it is Noetherian, since $\vF_q[T]$ is. Now, consider the sequence of $\vF_q[T]$-submodules $M_i=\langle 1,u,\ldots,u^i \rangle_{\vF_q[T]}$. This stabilizes at some $j$, giving $M_j=M_{j-1}$, which forces $u^j\in\langle 1,u,\ldots,u^{j-1} \rangle_{\vF_q[T]}$, proving the integrality of $u$.
\end{proof}

\begin{definition}
An element $\alpha \in \ker( \iota(\phi_a)) \subseteq \overline{k}$ is called an \emph{$a$-torsion point of $\phi$}. The set of all $a$-torsion points of $\phi$ is denoted by $\phi[a]$.
\end{definition}

Let $\mathfrak{l}\neq\mathfrak{p}$ be a prime of $A$, $m\geq 1$ be an integer and consider $\phi[\mathfrak{l}^m]$. Then the action by $\phi_\mathfrak{l}$ gives natural surjective homomorphisms
\[
\phi[\mathfrak{l}^{m+1}] \xrightarrow{\phi_\mathfrak{l}} \phi[\mathfrak{l}^m], \quad  \alpha\mapsto\phi_\mathfrak{l}(\alpha).
\]
\begin{definition}
The inverse limit
\[T_\mathfrak{l}(\phi) \coloneqq \mathop{\lim_{\longleftarrow}}_{m}\phi[\mathfrak{l}^m]
\]
with respect to the family $(\phi[\mathfrak{l^m}],\phi_\mathfrak{l})_{m\geq 1}$, is called the \emph{$\mathfrak{l}$-adic Tate module} of $\phi$.  
\end{definition}

Let $A_\mathfrak{l}$ be the completion of $A$ at $\mathfrak{l}$. Then, because of \cite[Theorem 3.5.2]{Papikian2023Drinfeld}, $T_\mathfrak{l}(\phi)$ is a free $A_\mathfrak{l}$-module of rank $r$, i.e. $T_\mathfrak{l}(\phi)\cong A_\mathfrak{l}^{\oplus r}$.
Now, let $u\in \End_k(\phi)$ and $\alpha\in\phi[\mathfrak{l^m}]$, then $\phi_{\mathfrak{l^m}}u(\alpha) = u\phi_{\mathfrak{l^m}}(\alpha)=0$. This means that $u$ induces an endomorphism on $\phi[\mathfrak{l^m}]$ for all $m\geq 1$, and hence an $A_\mathfrak{l}$-linear map
$u_\mathfrak{l}:T_\mathfrak{l}(\phi) \rightarrow T_\mathfrak{l}(\phi)$.

\begin{theorem}
Let $\phi$ be a Drinfeld module over $k$, and let $\mathfrak{l} \neq \mathfrak{p}$ be a prime of $A$. The natural map
\[
\End_k(\phi)\otimes_A A_\mathfrak{l} \rightarrow \End_{A_\mathfrak{l}}(T_\mathfrak{l}(\phi)), \qquad u\mapsto u_\mathfrak{l}
\]
is injective. Moreover the cokernel of this homomorphism is torsion-free.
\end{theorem}
\begin{proof}
    See \cite[Theorem 3.5.4]{Papikian2023Drinfeld}.
\end{proof}

Because of the last theorem, we can associate with every endomorphism $u\in\End_k(\phi)$ a matrix $u_\mathfrak{l}\in \Mat_{r\times r}(A_\mathfrak{l})$ that contains most of the relevant information about $u$. Therefore it makes sense now to define the following.
\begin{definition}
For an endomorphism $u\in\End_k(\phi)$, the \emph{characteristic polynomial} of $u$ is $P_u(X)\coloneqq \det(x-u_\mathfrak{l})$.
\end{definition}
As a consequence of \cite[Theorem 4.2.5]{Papikian2023Drinfeld}, we know that $P_u(X)$ is a polynomial with coefficients in $A$ which do not depend on the choice of $\mathfrak{l}$.

\section{Characteristic polynomial of the Frobenius endomorphism}

Let $m>0$ be an integer, $p$ a prime number, $q$ a $p$-power, and $\vF_{q^m}$ the finite field with $q^m$ elements. Recall that a linearized polynomial is an element $L=\sum_{i=0}^r t_iZ^{q^i}\in\vF_{q^m}[Z]$. We suppose that $t_0\in\vF_{q^d}$ for a $d|m$, and $\vF_q(t_0)=\vF_{q^d}$. Let $0<\ell\in\N$ be a positive integer and consider now the Drinfeld module attached to $L$ over the extension $k'=\vF_{q^{m\ell}}$ defined by
\begin{align}\label{eq:PhiEll}
\phi^{(\ell)}:\vF_q[T] & \rightarrow k'\{\tau\}\nonumber \\
T & \mapsto \phi_T^{(\ell)}=\sum_{i=0}^r t_i\tau^i.
\end{align}

Since $\gamma(T)=t_0$, then the $\vF_q[T]$-characteristic of $k'$ is $\mathfrak{p} = (m_{t_0})$ where $m_{t_0}$ is the minimal polynomial of $t_0$ over $\vF_q$. Let $\pi=\tau^{m\ell}$ be the Frobenius for $k'$, let $P_{\phi^{(\ell)}}\coloneqq P_{\phi^{(\ell)}}(X)$ be the characteristic polynomial of $\pi$ when acting on the $\mathfrak{l}$-adic Tate module $T_\mathfrak{l}(\phi^{(\ell)})$ of $\phi^{(\ell)}$. From \cite[Theorem 4.2.2]{Papikian2023Drinfeld} we know that $P_{\phi^{(\ell)}}$ has coefficients in $\vF_q[T]$ which do not depend on the choice of the prime $\mathfrak{l}\in\vF_q[T]$. Let $Q_{\phi^{(\ell)}}\coloneqq Q_{\phi^{(\ell)}}(X)$ be the minimal polynomial of $\pi$ over $\vF_q(T)$. Since we know that $\pi$ is integral over $\vF_q[T]$ (see Proposition \ref{theo:integral}), then $Q_{\phi^{(\ell)}}$ is also the minimal polynomial of $\pi$ over $\vF_q[T]$, and $[\vF_q(T,\pi):\vF_q(T)] = \rank_{\vF_q[T]}\vF_q[T,\pi]$. As a result of the theory (see \cite[Theorem 4.2.5]{Papikian2023Drinfeld}), we have
\[
P_{\phi^{(\ell)}} = Q_{\phi^{(\ell)}}^{\frac{r}{[\vF_q(T,\pi):\vF_q(T)]}}.
\]

Let $K$ be a field. For a polynomial $F\in K[X]$, let $F=\prod_{i=1}^{\deg(F)}(X-\alpha_i)$ be its factorization into linear factors (with possible repeated roots) in $\bar{K}[X]$, where $\bar{K}$ is the algebraic closure of $K$. Let $\epsilon_\ell:K[X]\rightarrow \bar{K}[X]$ be the map defined by $\epsilon_\ell(F)=\prod_{i=1}^{\deg(F)}(X-\alpha_i^\ell)$.

Sometimes we need to distinguish the Frobenius for different values of $\ell$. Therefore, we write $\pi_\ell=\tau^{m\ell}$, and $d_\ell = \deg Q_{\phi^{(\ell)}} = [\vF_q(T,\pi_\ell):\vF_q(T)]$. We also denote by $S_\ell\supseteq \vF_q(T)$ the splitting field of $Q_{\phi^{(\ell)}}$, and by $A_\ell \coloneqq \Aut(S_\ell:\vF_q(T))$ its group of automorphisms.

\begin{proposition}\label{prop:QEllDividesEpsEll}
Let $0<m\in\N$ and $0<\ell\in\N$. Let $\phi^{(\ell)}$ be the Drinfeld module attached to a linearized polynomial $L$ over the extension $\vF_{q^{m\ell}}$.
Let $\pi_\ell=\tau^{m\ell}$, in particular $\pi_1=\tau^m$. Let $Q_{\phi^{(\ell)}}$ be the minimal polynomial of $\pi_\ell$ over $\vF_q(T)$ and $d_\ell = \deg Q_{\phi^{(\ell)}} = [\vF_q(T,\pi_\ell):\vF_q(T)]$ its degree. Let $\epsilon_\ell:K[X]\rightarrow \bar{K}[X]$ be the map defined by $\epsilon_\ell(F)=\prod_{i=1}^{\deg(F)}(X-\alpha_i^\ell)$. Then $Q_{\phi^{(\ell)}}$ divides $\epsilon_\ell(Q_{\phi^{(1)}})$.
\end{proposition}
\begin{proof}
First of all, notice that the map $\epsilon_\ell$ does not change the degree of the polynomial. Moreover, since $Q_{\phi^{(1)}}$ is a polynomial over $\vF_q[T]$, then even $\epsilon_\ell(Q_{\phi^{(1)}})\in\vF_q[T][X]$. To see this, suppose $\alpha_1,\ldots,\alpha_{d_1}\in S_1$ are all the roots of $Q_{\phi^{(1)}}$, with possible repetitions (these polynomials might a priori be inseparable). Then we can consider the factorization into linear factors of $Q_{\phi^{(1)}}$ in its splitting field $S_1$:
\begin{equation}\label{eq:QPhi1}
Q_{\phi^{(1)}} = \prod_{i=0}^{d_1}(X-\alpha_i).
\end{equation}
On the other hand, if $s_1=s_1(X_1,\ldots,X_{d_1}),\ldots,s_{d_1}=s_{d_1}(X_1,\ldots,X_{d_1})$ are the $d_1$ elementary symmetric polynomials in the unknowns $X_1,\ldots,X_{d_1}$, we can write
\[
Q_{\phi^{(1)}} = X^{d_1}-s_1(\alpha_1,\ldots,\alpha_{d_1})X^{d_1-1}+\ldots+(-1)^{d_1}s_{d_1}(\alpha_1,\ldots,\alpha_{d_1}),
\]
and since $Q_{\phi^{(1)}}\in\vF_q[T][X]$ the evaluation at $\alpha=(\alpha_1,\ldots,\alpha_{d_1})\in S_1^{d_1}$ defines a homomorphism over the $\vF_q$-algebra generated by $s_1,\ldots,s_{d_1}$:
\begin{align*}
   \ev_\alpha :  \vF_q[s_1,\ldots,s_{d_1}] & \rightarrow \vF_q[T] \\
   f(X_1,\ldots,X_{d_1}) & \mapsto \ev_\alpha(f)=f(\alpha_1,\ldots,\alpha_{d_1}).
\end{align*}
Now, applying the map $\epsilon_\ell$ to the Equation \eqref{eq:QPhi1}, we immediately see that $\alpha_1^\ell,\ldots,\alpha_{d_1}^\ell$ are all the roots of $\epsilon_\ell(Q_{\phi^{(1)}})$, with possible repetitions. In terms of symmetric polynomials we have
\[
\epsilon_\ell(Q_{\phi^{(1)}}) = X^{d_1}-s_1(\alpha_1^\ell,\ldots,\alpha_{d_1}^\ell)X^{d_1-1}+\ldots+(-1)^{d_1}s_{d_1}(\alpha_1^\ell,\ldots,\alpha_{d_1}^\ell).
\]
By the Fundamental Theorem of Symmetric Polynomials (see \cite[Chapter 5, Proposition 2.20]{Hungerford2003Algebra}), since for $i=1,\ldots,d_1$ each polynomial $s_i(X_1^\ell,\ldots,X_{d_1}^\ell)$ is symmetric, it belongs to the $\vF_q$-algebra generated by the elementary symmetric polynomials $s_1,\ldots,s_{d_1}$. This implies that 
\begin{equation}\label{eq:Sym}
    s_i(\alpha_1^\ell,\ldots,\alpha_{d_1}^\ell)=\ev_\alpha(s_i(X_1^\ell,\ldots,X_{d_1}^\ell)) \in\vF_q[T], 
\end{equation}
namely $\epsilon_\ell(Q_{\phi^{(1)}})\in\vF_q[T][X]\subset \vF_q(T)[X]$.

Finally, since $\pi_1$ is a root of $Q_{\phi^{(1)}}$, then $\pi_\ell$ is a root of $\epsilon_\ell(Q_{\phi^{(1)}})$, but on the other hand $\pi_\ell$ is a root of its own minimal polynomial $Q_{\phi^{(\ell)}}$. By minimality, $Q_{\phi^{(\ell)}}$ must divide $\epsilon_\ell(Q_{\phi^{(1)}})$.
\end{proof}

\begin{remark}\label{rem:QEllQ1Split}
Notice that since $Q_{\phi^{(\ell)}}$ divides $\epsilon_\ell(Q_{\phi^{(1)}})$, then the splitting field $S_\ell$ is contained in the splitting field of $\epsilon_\ell(Q_{\phi^{(1)}})$ which, in turn, is contained in $S_1$. Recall that $A_i=\Aut(S_i:\vF_q(T))$.
Therefore, we have that the restriction of a automorphisms to $S_\ell$ gives a surjective map $A_1\twoheadrightarrow A_\ell$ whose kernel has size $|A_1|/|A_\ell|$, namely every automorphism of $A_\ell$ can be extended to $|A_1|/|A_\ell|$ automorphisms in $A_1$ (see \cite[Chapter 5, Theorem 3.8]{Hungerford2003Algebra}).
\end{remark}

\begin{proposition}\label{prop:EpsEllPowerQEll}
Let $0<m\in\N$ and $0<\ell\in\N$. Let $\phi^{(\ell)}$ be the Drinfeld module attached to a linearized polynomial $L$ over the extension $\vF_{q^{m\ell}}$.
Let $\pi_\ell=\tau^{m\ell}$, in particular $\pi_1=\tau^m$. Let $Q_{\phi^{(\ell)}}$ be the minimal polynomial of $\pi_\ell$ over $\vF_q(T)$ and $d_\ell = \deg Q_{\phi^{(\ell)}} = [\vF_q(T,\pi_\ell):\vF_q(T)]$ its degree. Let $\epsilon_\ell:K[X]\rightarrow \overline K[X]$ be the map defined by $F\mapsto \epsilon_\ell(F)=\prod_{i=1}^{\deg(F)}(X-\alpha_i^\ell)$. Then the following equality holds:
\[
\epsilon_\ell(Q_{\phi^{(1)}}) = Q_{\phi^{(\ell)}}^\frac{d_1}{d_\ell}.
\]
\end{proposition}
\begin{remark}
Notice that this proposition is pretty much automatic when the extension is Galois. Since $Q_{\phi^{(\ell)}}$ might be inseparable, we need a more careful analysis.
Also, notice that it is immediate to see that the target space of $\epsilon_\ell$ can be restricted to $K[X]$ simply because if $s(x_1,\dots,x_N)$ is a symmetric polynomial, so is $s(x_1^\ell,\dots,x_N^\ell)$.
\end{remark}
\begin{proof}[Proof of Proposition \ref{prop:EpsEllPowerQEll}]

For every $\ell\in\N$, since $Q_{\phi^{(\ell)}}$ is irreducible, all its roots have the same multiplicity, say $u_\ell$, dividing $d_\ell$. Consider now the factorization of $Q_{\phi^{(1)}}$ in its splitting field $S_1$, then $A_1$ acts transitively on its roots. This means that we can find $\sigma_1,\ldots,\sigma_{\frac{d_1}{u_1}}\in A_1$ pairwise distinct for which
\[Q_{\phi^{(1)}} = \prod_{i=1}^{\frac{d_1}{u_1}}(X-\sigma_i(\pi_1))^{u_1}.\]

This obviously implies that
\[
\epsilon_\ell(Q_{\phi^{(1)}}) = \prod_{i=1}^{\frac{d_1}{u_1}}(X-\sigma_i(\pi_\ell))^{u_1}.
\]

In the same way we can find $\delta_1,\ldots,\delta_{\frac{d_\ell}{u_\ell}}\in A_\ell$ such that $\delta_1(\pi_\ell),\ldots,\delta_{\frac{d_\ell}{u_\ell}}(\pi_\ell)$ are pairwise distinct, and for which
\[
Q_{\phi^{(\ell)}} = \prod_{i=1}^{\frac{d_\ell}{u_\ell}}(X-\delta_i(\pi_\ell))^{u_\ell}.
\]

Now let us define the polynomial
\[
D_\ell = \prod_{\lambda\in A_\ell}(X-\lambda(\pi_\ell))^{u_\ell}.
\]

Obviously $Q_{\phi^{(\ell)}} | D_\ell$. However, since $\delta_1(\pi_\ell),\ldots,\delta_{\frac{d_\ell}{u_\ell}}(\pi_\ell)\in A_\ell$ are pairwise distinct, they form a complete set of representatives for the cosets in $A_\ell$ of the stabilizer of $\pi_\ell$. By Orbit-Stabilizer Theorem, this means that for every $\delta_i$ we can find $u_\ell|A_\ell|/d_\ell$ automorphisms in $A_\ell$ whose action on $\pi_\ell$ is the same as the action of $\delta_i$. This implies that not only $Q_{\phi^{(\ell)}}$ divides $D_\ell$, but also:
\begin{equation}\label{eq:DEll}
D_\ell = Q_{\phi^{(\ell)}}^\frac{u_\ell|A_\ell|}{d_\ell}.
\end{equation}
In particular we have
\[
D_1 = Q_{\phi^{(1)}}^\frac{u_1|A_1|}{d_1}
\]
and if we compute $\epsilon_\ell(D_1)^{u_\ell}$ we get
\begin{equation}\label{eq:EpsEll1}
\epsilon_\ell(D_1)^{u_\ell} = \epsilon_\ell(Q_{\phi^{(1)}}^\frac{u_1|A_1|}{d_1})^{u_\ell} =  (\epsilon_\ell(Q_{\phi^{(1)}}))^\frac{u_1u_\ell|A_1|}{d_1}.
\end{equation}

On the other hand, as a consequence of the Remark~\ref{rem:QEllQ1Split}, $|A_1|/|A_\ell|$ automorphisms of $A_1$ have the same action on $S_\ell$, and since $\pi_\ell\in S_\ell$ we have

\begin{align}\label{eq:EpsEll2}
\begin{split}
\epsilon_\ell(D_1)^{u_\ell} & 
= \epsilon_\ell \Big( \prod_{\sigma\in A_1}(X-\sigma(\pi_1))^{u_1}\Big)^{u_\ell} = \prod_{\sigma\in A_1}(X-\sigma(\pi_\ell))^{u_1u_\ell} = \\
& = \Big( \prod_{\lambda\in A_\ell}(X-\lambda(\pi_\ell)) \Big)^\frac{u_1u_\ell|A_1|}{|A_\ell|} = D_\ell^\frac{u_1|A_1|}{|A_\ell|}.
\end{split}
\end{align}

Finally, combining Equations \eqref{eq:EpsEll1}, \eqref{eq:EpsEll2} and \eqref{eq:DEll} we get
\[
(\epsilon_\ell(Q_{\phi^{(1)}}))^\frac{u_1u_\ell|A_1|}{d_1} = D_\ell^\frac{u_1|A_1|}{|A_\ell|} = Q_{\phi^{(\ell)}}^\frac{u_1u_\ell|A_1|}{d_\ell}.
\]

Raising both sides to $d_1$ and using the fact that both polynomials are monic, we prove our claim.
\end{proof}

\begin{corollary}\label{cor:PEllEqEpsEll}
Let $0<m\in\N$ and $0<\ell\in\N$. Let $\phi^{(\ell)}$ be the Drinfeld module attached to a linearized polynomial $L$ over the extension $\vF_{q^{m\ell}}$.
Let $\pi_\ell=\tau^{m\ell}$, in particular $\pi_1=\tau^m$. Let $Q_{\phi^{(\ell)}}$ be the minimal polynomial of $\pi_\ell$ over $\vF_q(T)$ and $d_\ell = \deg Q_{\phi^{(\ell)}} = [\vF_q(T,\pi_\ell):\vF_q(T)]$ be its degree. Let $P_{\phi^{(\ell)}}$ be the characteristic polynomial of $\pi_\ell$ when acting on the $\mathfrak{l}$-adic Tate module $T_\mathfrak{l}(\phi^{(\ell)})$ of $\phi^{(\ell)}$. Let $\epsilon_\ell:K[X]\rightarrow \overline K[X]$ be the map defined by $F\mapsto \epsilon_\ell(F)=\prod_{i=1}^{\deg(F)}(X-\alpha_i^\ell)$. Then the following equality holds:
\[
P_{\phi^{(\ell)}} = \epsilon_\ell(P_{\phi^{(1)}}).
\]
\end{corollary}
\begin{proof}
The proof is a direct computation that follows from the definition of the characteristic polynomial and Proposition~\ref{prop:EpsEllPowerQEll}:
\[
\epsilon_\ell(P_{\phi^{(1)}}) = \epsilon_\ell(Q_{\phi^{(1)}}^\frac{r}{d_1}) = (\epsilon_\ell(Q_{\phi^{(1)}}))^\frac{r}{d_1} = (Q_{\phi^{(\ell)}}^\frac{d_1}{d_\ell})^\frac{r}{d_1} = Q_{\phi^{(\ell)}}^\frac{r}{d_\ell} = P_{\phi^{(\ell)}}.
\]
\end{proof}

\section{Linear Recurrence Sequences and Characteristic polynomial of a Linearized polynomial}

\subsection{Linear Recurrence Sequences}
For a complete treatment about linear recurrence sequences see for example \cite{Everest2003Recurrence}. If not specified otherwise, we denote by $\vR$ a commutative ring with identity, by $\vR[X]$ the ring of polynomials in the indeterminate $X$ with coefficients over $\vR$, and by $\vR^{\langle 1\rangle}$ the set of all the sequences $\{a_n\}_{n=0}^\infty\subset\vR$. 

\begin{definition}\label{def:LinearRec1}
A sequence $\{a_n\}_{n=0}^\infty\in\vR^{\langle 1\rangle}$ is called \emph{linear recurrence sequence} over $\vR$ if, for any $n \geq 1$, it satisfies a homogeneous linear recurrence relation of the form
\begin{equation}\label{eq:LinearSeq}
    a_{n+d} = c_{d-1} a_{n+d-1} + \ldots + c_1 a_{n+1} + c_0 a_n,
\end{equation}
with fixed integer $d\geq 1$ and constant coefficients $c_0,\ldots,c_{d-1}\in\vR$. The integer $d$ is called the \emph{order of the relation}. We denote by $LRS(\vR)$ the set of all linear recurrence sequences over $\vR$.  
\end{definition}

\begin{definition}
The polynomial associated to the relation \eqref{eq:LinearSeq} is
    \[f(X)=X^d-c_{d-1}X^{d-1}-\ldots-c_1X-c_0\in\vR[X]\]
and it is called the \emph{characteristic polynomial} of the linear recurrence relation. 
\end{definition}

If $\vR$ is a ring without zero divisors, then any linear recurrence sequence satisfies a recurrence relation of minimal length. The characteristic polynomial of the minimal length relation is called the \emph{minimal polynomial} of the sequence, and any characteristic polynomial is divisible by such minimal polynomial.
\begin{definition}
    The \emph{order} of a linear recurrence sequence is the degree of the minimal polynomial of the sequence.
\end{definition}

From now on, to simplify the notation, we could sometimes write $\{a_n\}$ for $\{a_n\}_{n=0}^\infty$.

\begin{remark}
We can define an action of $\vR[X]$ on $\vR^{\langle 1\rangle}$. Indeed, given a polynomial $g(X)=\sum_{i=0}^m r_i X^i \in\vR[X]$ and a sequence $\{a_n\}_{n\in \vN}\in \vR^{\langle 1\rangle}$, define the action by
\[
g(X)\cdot\{a_n\}_{n\in \vN}\coloneqq\{b_n\}_{n\in \vN}=\left\{\sum_{i=0}^m r_i a_{n+i}\right\}_{n\in \vN}\in \vR^{\langle 1\rangle}
\]
It is easy to verify that for $f(X),g(X)\in \vR[X]$ 
\[(f(X)g(X))\cdot \{a_n\}_{n\in \vN}=f(X)\cdot(g(X)\cdot\{a_n\}_{n\in \vN}).\]
This gives to $\vR^{\langle 1\rangle}$ an $\vR[X]$-module structure. Note that a sequence $\{a_n\}_{n\in \vN}$ is a linear recurrence sequence if and only if it is annihilated by some nonzero monic polynomial $f(X)\in\vR[X]$, i.e., $f(X)\cdot\{a_n\}_{n\in \vN}=0$. 
\end{remark}

By definition, recall that if $\{a_n\}_{n=0}^\infty\in LRS(\vR)$ then it has characteristic polynomials defined over $\vR$. It is well know in literature that the sum and the product of two linear recurrence sequences defined over a field $F$ are still linear recurrence sequences over $F$ (see for example \cite[Theorem 4.1]{Everest2003Recurrence}). Nevertheless, when we consider sequences over an arbitrary ring $\vR$, such results are generally not true. However, many of the expected results still hold when $\vR$ is a commutative ring with identity. The following is an easy exercise.

\begin{proposition}\label{prop:LinearAddition}
Let $\vR$ be a commutative ring with identity. Let $\{a_n\}_{n=0}^\infty, \{b_n\}_{n=0}^\infty\in LRS(\vR)$ of order $s$ and $t$ respectively. Then $\{c_n=a_n+b_n\}_{n=0}^\infty$ is linear recurrence sequence over $\vR$ of order at most $s+t$. In other words, the set $LRS(\vR)$ of linear recurrence sequences over $\vR$ is a submodule of $\vR^{\langle 1\rangle}$.
\end{proposition}
\begin{proof}
    Suppose $f(X),g(X)\in\vR[X]$ are two minimal polynomials for $\{a_n\}$ and $\{b_n\}$ respectively. In particular, this means that $f(X)\cdot\{a_n\}=0$ and $g(X)\cdot\{b_n\}=0$. Now simply observe that the product $f(X)g(X)\in\vR[X]$ annihilates the sequence $a_n+b_n$. 
%    We have:
%    \begin{align*}
%    f(X)g(X)\cdot\{c_n\} & = f(X)\cdot(g(X)\cdot\{a_n+b_n\}) = f(X)\cdot(g(X)\cdot\{a_n\}+g(X)\cdot\{b_n\}) \\
%    & = f(X)\cdot(g(X)\cdot\{a_n\}) = f(X)g(X)\cdot\{a_n\} = g(X)f(X)\cdot\{a_n\}=0
%   % & = g(X)\cdot(f(X)\cdot\{a_n\}) = g(X)\cdot\{0\} = 0
%    \end{align*}
%    where the last equality follows by the commutativity of $\mathcal R[x]$.
    This means that $f(X)g(X)\in\vR[X]$ is a characteristic polynomial for the sequence $\{a_n+b_n\}_{n=0}^\infty$. Since $f(X)g(X)$ has degree $s+t$, the claim is proved.
\end{proof}

We now prove a lemma that simplifies the notation in the proof of Corollary~\ref{cor:CEllLinearSeq}.
\begin{lemma}\label{lemma:SeqProd}
    Let $F$ be a field, and $\vR=F[T]$ be the ring of polynomials in $T$ with coefficients over $F$. Let $\{a_n\}_{n=0}^\infty\in LRS(\vR)$ be of order $d>0$, and let $\{b_n\}_{n=0}^\infty\in LRS(F)$ be of order 1. Then the product sequence $\{b_na_n\}_{n=0}^\infty$ belongs to $LRS(\vR)$ and it is of order $d$.
\end{lemma}
\begin{proof}
    Write minimal recurrence relations for $\{a_n\}_{n=0}^\infty$ and $\{b_n\}_{n=0}^\infty$ as follows:
    \[
    \begin{gathered}
        a_{n+d} = c_{d-1} a_{n+d-1} + \ldots + c_1 a_{n+1} + c_0 a_n ; \\
        b_{n+1} = \beta b_{n}
    \end{gathered}
    \]
    where $c_{d-1},\ldots,c_0\in F[T]$ and $\beta\in F$. Now, expanding the product $b_{n+d}a_{n+d}$, using repeatedly the relation for $\{b_n\}$, we get
    \begin{align}\label{eq:SameCoeff}
    \begin{split}
        b_{n+d}a_{n+d} & = \beta b_{n+d-1}a_{n+d} = \beta b_{n+d-1}(c_{d-1} a_{n+d-1} + \ldots + c_1 a_{n+1} + c_0 a_n)  \\
        & =  \beta b_{n+d-1}c_{d-1}a_{n+d-1} + \beta b_{n+d-1}c_{d-2}a_{n+d-2}+\ldots + \beta b_{n+d-1} c_0a_{n} \\
        & = \beta b_{n+d-1}c_{d-1}a_{n+d-1} + \beta^2 b_{n+d-2}c_{d-2}a_{n+d-2}+\ldots + \beta^d b_{n} c_0a_{n} \\
        & = (\beta c_{d-1})b_{n+d-1}a_{n+d-1} + (\beta^2 c_{d-2})b_{n+d-2}a_{n+d-2} + \ldots + (\beta^d c_0)b_n a_n.
    \end{split}
    \end{align}
   This simply means that the product sequence $\{b_na_n\}_{n=0}^\infty\subseteq  \vR$ is a linear recurrence sequence of order at most $d$. Suppose now its order is $r<d$. Then we could write
   \[
    b_{n+r}a_{n+r} =  \bar{c}_{r-1} b_{n+r-1}a_{n+r-1} + \ldots + \bar{c}_0 b_{n}a_{n}
   \]
   for some $\bar{c}_{r-1},\ldots,\bar{c}_0\in F[T]$. Since $F$ is a field, we could divide by $b_{n+r}$ and obtain
   \begin{align*}
   a_{n+r} & =  \bar{c}_{r-1} \frac{b_{n+r-1}}{b_{n+r}}a_{n+r-1} + \ldots + \bar{c}_0 \frac{b_{n}}{b_{n+r}}a_{n}  \\
   & = \frac{\bar{c}_{r-1}}{\beta}a_{n+r-1} + \ldots + \frac{\bar{c}_0}{\beta^r}a_{n}
   \end{align*}
   that is a recurrence relation for $\{a_n\}_{n=0}^\infty\subseteq F[T]$ of order $r<d$. This contradicts the minimality of $d$, therefore it must be $r=d$.
\end{proof}

\begin{remark}\label{rem:LinearRecurNthTerm}
Let $\{a_n\}_{n=0}^\infty$ be a linear recurrence sequence satisfying a homogeneous linear recurrence relation of order $d$ as in Equation~\eqref{eq:LinearSeq}. There is a simple method to compute the term $a_{n+d}$ of the sequence (see for example \cite{miller1966algorithm}, \cite{bostan2021simple}). For any $n\in \mathbb{N}$ define the \textit{$n$-th state vector} to be the column vector 
\[U_n\coloneqq [a_{n+d-1},a_{n+d-2},\ldots,a_n]^\top.\]
Consider the matrix 
\begin{equation}\label{eq:matrix}
    M \coloneqq \begin{bmatrix}
    c_{d-1} & c_{d-2} & \dots & c_1 & c_0 \\
    1 & 0 & \dots & 0 & 0 \\
    \vdots & \vdots & \ddots & \vdots & \vdots \\
    0 & 0 & \dots & 1 & 0
\end{bmatrix}\in\Mat_d(\vR).
\end{equation}
It is straightforward to check that 
\[
M^{n+1} U_0 = U_{n+1}.
\]
The $i$-th power of the matrix $M$ can be computed efficiently using square and multiply
\[
M^i =
\begin{cases}
    (M^{i/2})^2, & \text{if } i \text{ is even}; \\
    M \cdot M^{i-1}, & \text{if } i \text{ is odd}.
\end{cases}
\]
This approach allows us to do the computation of  $M^{n+1}$ in $O(\log(n+1))=O(\log(n))$ matrix multiplications via exponentiation by squaring, i.e. $O(d^3\log(n+1))=O(d^3\log(n))$ multiplications over $\vR$. Consequently, given $U_0$ and $M$, one can obtain $U_{n+1}$ in $O(d^3\log(n+1)+d^2)=O(d^3\log(n))$ multiplications over $\vR$. Finally, we get the term $a_{n+d}$ of the linear recurrence sequence as the first entry of $U_{n+1}$, with a computational cost of $O(d^3\log n)$ over $\vR$.

\end{remark}

\subsection{Characteristic polynomial of a linearized polynomial}
We are now ready  to provide the theoretical foundations of our algorithm, which heavily rely on the theory of Drinfeld modules.

As before, take $\ell \in \N_{>0}$, and consider a linearized polynomial over $k=\vF_{q^m}$ of the form
\begin{equation}\label{eq:LinearPoly}
L=\sum_{i=0}^r t_iZ^{q^i}\in \vF_{q^m}[Z].
\end{equation}
For $k'=\vF_{q^{m\ell}}$ extension of $\vF_q$, denote by $C_L^{(\ell)}\in \vF_q[T]$ the characteristic polynomial of $L$ seen as linear map over $k'$. To determine $C_L^{(\ell)}$, we use the Drinfeld module $\phi^{(\ell)}$ attached to $L$ over the extension $k'$, as defined by Equation~\eqref{eq:PhiEll}. Recall that $P_{\phi^{(\ell)}}\coloneqq P_{\phi^{(\ell)}}(X)$ is the characteristic polynomial of the Frobenius $\pi_\ell\coloneqq \tau^{m\ell}$ when acting on the $\mathfrak{l}$-adic Tate module $T_\mathfrak{l}(\phi^{(\ell)})$ of $\phi^{(\ell)}$.

\begin{remark}\label{rem:FittingIdeal}
In general, if $\phi$ is a Drinfeld module defined over an $\vF_q[T]$-field $k$, let $^\phi k$ denote the $\vF_q[T]$-module, whose underlying group is $(k,+)$ subject to the action defined by $a\circ\beta \coloneqq \phi_a(\beta)$ for every $a\in\vF_q[T]$ and $\beta\in k$. Since $^\phi k$ is finitely generated, then we have an isomorphism
\[
^\phi k \cong \vF_q[T]/{(a_1)} \oplus\ldots\oplus\vF_q[T]/{(a_s)}
\]
for uniquely determined monic polynomials $a_1,\ldots,a_s\in\vF_q[T]$ of positive degrees such that $a_i\mid a_{i+1}$. Each $a_i$ is called an \emph{invariant factor} of $^\phi k$, and the product of all the invariant factors $\chi(^\phi k)\coloneqq\prod_{i=1}^s a_i$ is called the \emph{fitting ideal} of $^\phi k$. As a standard linear algebra fact (see for example \cite[Exercise 1.2.7]{Papikian2023Drinfeld} or \cite[Chapter 12, Proposition 20]{dummit2004abstract}), we also know that $\chi(^\phi k)$ is the characteristic polynomial of $\phi_T$ considered as an $\vF_q$-linear map over $k$. This means that, setting $\phi=\phi^{(\ell)}$ and $k=k'$, we have $\chi(^\phi k) = C_L^{(\ell)}$.

Finally, as a consequence of \cite[Theorem 4.2.6]{Papikian2023Drinfeld}, observe that $(P_{\phi^{(\ell)}}(1))=(\chi(^\phi k))=(C_L^{(\ell)})$ which implies, using Corollary~\ref{cor:PEllEqEpsEll}, that $C_L^{(\ell)}=v^{(\ell)}P_{\phi^{(\ell)}}(1)=v^{(\ell)}\epsilon_\ell(P_{\phi^{(1)}})(1)$ for some $v^{(\ell)}\in \vF_q$. 
\end{remark}

\begin{remark}\label{rem:NormCoeff}
The constant term $v^{(\ell)}\in \vF_q$ in the previous Remark~\ref{rem:FittingIdeal} can be computed explicitly. Indeed, since $C_L^{(\ell)}$ is a monic polynomial in $\vF_q[T]$, $v^{(\ell)}$ must be the inverse of the leading coefficient of $P_{\phi^{(\ell)}}(1)\in \vF_q[T]$. Therefore, using \cite[Theorem 4.2.7]{Papikian2023Drinfeld}, it must be
\[
v^{(\ell)}=(-1)^{(r-1)m\ell-r}\Nr_{k'/\vF_q}(t_r),
\]
where $t_r$ is the leading coefficient of $L$.
Since $t_r\in \vF_{q^m}$, and $\vF_q\subset \vF_{q^m}\subset\vF_{q^{m\ell}}=k'$, the following holds:
\[
\Nr_{k'/\vF_q}(t_r)=\Nr_{\vF_{q^m}/\vF_q}(\Nr_{k'/\vF_{q^m}}(t_r))
=\Nr_{\vF_{q^m}/\vF_q}(t_r^{\ell})
=(\Nr_{\vF_{q^m}/\vF_q}(t_r))^\ell.
\]
To simplify the notation, let $N\coloneqq\Nr_{\vF_{q^m}/\vF_q}(t_r)\in \vF_q$. Then we have 
\[
v^{(\ell)}=(-1)^{(r-1)m\ell-r}N^\ell.
\]
In particular, a simple computation shows that $v^{(\ell)}=(-1)^{(r-1)m}Nv^{(\ell-1)}$, which implies that the sequence $\{v^{(\ell)}\}_{\ell\in\N_{>0}}$ is a linear recurrence sequence of order 1.
\end{remark}

We have now all the tools to show that $\{C_L^{(\ell)}\}_{\ell\in\N_{>0}}$ is a linear recurrence sequence. We start with a preliminary proposition.

\begin{proposition}\label{prop:PEll1LinearSeq}
Let $L$ be a linearized polynomial of $q$-degree $r$ as given in Equation~\eqref{eq:LinearPoly}. Let $\ell\in\N_{>0}$ and $\phi^{(\ell)}$ be the Drinfeld module attached to $L$ over the extension $\vF_{q^{m\ell}}$ as defined in Equation~\eqref{eq:PhiEll}.
Let $\pi_\ell=\tau^{m\ell}$ be the Frobenius endomorphism. Let $P_{\phi^{(\ell)}}\coloneqq P_{\phi^{(\ell)}}(X)$ be the characteristic polynomial of $\pi_\ell$ when acting on the $\mathfrak{l}$-adic Tate module $T_\mathfrak{l}(\phi^{(\ell)})$ of $\phi^{(\ell)}$. Then the sequence $\{P_{\phi^{(\ell)}}(1)\}_{\ell\in\N_{>0}}$ is a linear recurrence sequence of order at most $2^r$ in $\vF_q[T]$.
\end{proposition}
\begin{proof}
Let $\beta_1,...,\beta_r\in\overline{\vF_q(T)}$ be all the roots (not necessarily distinct) of $P_{\phi^{(1)}}$. Let also $s_1(X_1,\ldots,X_{r}), \ldots, s_r(X_1,\ldots,X_{r})$ be the $r$ elementary symmetric polynomials in the unknowns $X_1,\ldots,X_{r}$. By Corollary~\ref{cor:PEllEqEpsEll} we can write
\[
P_{\phi^{(\ell)}}(X) = \prod_{i=1}^r(x-\beta_i^\ell) = X^r-s_1(\beta_1^\ell,\ldots,\beta_r^\ell)X^{r-1}+\ldots+(-1)^rs_r(\beta_1^\ell,\ldots,\beta_r^\ell),
\]
and evaluating at 1 we get
\begin{equation}\label{eq:PEll1}
P_{\phi^{(\ell)}}(1) = 1-s_1(\beta_1^\ell,\ldots,\beta_r^\ell)+\ldots+(-1)^rs_r(\beta_1^\ell,\ldots,\beta_r^\ell).
\end{equation}
Notice that for any $\ell>0$ and $i=1,\ldots,r$ we have $s_i(\beta_1^\ell,\ldots,\beta_r^\ell)\in\vF_q[T]$.

Now, for $\ell\geq r$, by Girard–Newton formula for symmetric polynomials (for reference see for example \cite[\S 12]{Edwards1984Galois} or \cite[p. 203]{newton1967mathematical}), the following holds:
\begin{equation}\label{eq:S_1Recurrence}
s_1(\beta_1^\ell,\ldots,\beta_r^\ell)=\sum_{j=1}^r(-1)^{j-1}s_j(\beta_1,\ldots,\beta_r)s_1(\beta_1^{\ell-j},\ldots,\beta_r^{\ell-j}). 
\end{equation}
Therefore $\{s_1(\beta_1^\ell,\ldots,\beta_r^\ell)\}_{\ell\in\N_{>0}} \subseteq \vF_q[T]$ is a linear recurrence sequence of order at most $r$ in $\vF_q[T]$, namely $\{s_1(\beta_1^\ell,\ldots,\beta_r^\ell)\}_{\ell\in\N_{>0}} \in LRS(\vF_q[T])$ .

On the other hand, for any $V\subset\{1,\ldots,r\}$ such that $|V|=i$, let $\beta_V\coloneqq\prod_{v\in V}\beta_v$. Let then $\overline{s_j}(\ldots,X_V,\ldots)$ be the $j$-th symmetric polynomial in $\binom{r}{i}$ variables, for $j=1,\ldots,\binom{r}{i}$. Then, for each fixed $i=2,\ldots,r$ we can write
\[s_i(\beta_1^\ell,\ldots,\beta_r^\ell)=\overline{s_1}(\ldots,\beta_V^\ell,\ldots).
\]           
Now, applying again the Girard–Newton formula for symmetric polynomials, we have 
\begin{equation}\label{eq:S_iRecurrence}
\overline{s_1}(\ldots,\beta_V^\ell,\ldots)
=\sum_{j=1}^{\binom{r}{i}}(-1)^{j-1}\overline{s_j}(\ldots,\beta_V,\ldots)\overline{s_1}(\ldots,\beta_V^{\ell-j},\ldots).
\end{equation}
However, notice that not only is $\overline{s_j}(\ldots,\beta_V,\ldots)$ invariant under permutations of the $\beta_V$'s, but it is also invariant under permutations of the $\beta_i$'s. This means that $\overline{s_j}(\ldots,\beta_V,\ldots) = f(\beta_1,\ldots,\beta_r)$ for a certain symmetric polynomial $f\in\vF_q[s_1,\ldots,s_r]$. Since the evaluation at the $\beta_i$'s is an $\vF_q$-algebra homomorphism, we get that $\overline{s_j}(\ldots,\beta_V,\ldots)\in\vF_q[T]$.

Therefore, each sequence $\{s_i(\beta_1^\ell,\ldots,\beta_r^\ell)\}_{\ell\in\N_{>0}}$ is also in $LRS(\vF_q[T])$ but of order at most $\binom{r}{i}$. Finally, notice that $\{1\}_{\ell\in\N_{>0}}$ is obviously a linear recurrence sequence of order 1. Using now Proposition~\ref{prop:LinearAddition} we easily see that $\{P_{\phi^{(\ell)}}(1)\}_{\ell\in\N_{>0}}$ is a linear recurrence sequence of order at most $\sum_{i=0}^r\binom{r}{i} = 2^r$.
\end{proof}

\begin{corollary}\label{cor:CEllLinearSeq}
Let $L$ be a linearized polynomial of $q$-degree $r$ as given in equation~\eqref{eq:LinearPoly}. Then the sequence $\{C_L^{(\ell)}\}_{\ell\in\N_{>0}}\subseteq \vF_q[T]$ of characteristic polynomials of $L$ as a linear map on $\vF_{q^{m\ell}}$ is a linear recurrence sequence of order at most $2^r$.   
\end{corollary}
\begin{proof}
From Remark~\ref{rem:FittingIdeal}, recall that 
\[
\{C_L^{(\ell)}\}_{\ell\in\N_{>0}}=\{v^{(\ell)}P_{\phi^{(\ell)}}(1)\}_{\ell\in\N_{>0}}.
\]
By Remark~\ref{rem:NormCoeff}, the sequence $\{v^{(\ell)}\}_{\ell\in\N_{>0}}$ satisfies a linear recurrence relation of order 1 over $\vF_q$, and by Proposition~\ref{prop:PEll1LinearSeq}, the sequence $\{P_{\phi^{(\ell)}}(1)\}_{\ell\in\N_{>0}}$ satisfies a linear recurrence relation of order $d\leq 2^r$ over $\vF_q[T]$. As a consequence of Lemma~\ref{lemma:SeqProd} we have that $\{C_L^{(\ell)}\}_{\ell\in\N_{>0}}\subseteq \vF_q[T]$ is a linear recurrence sequence of order exactly $d$.
\end{proof}

The upper bound for the order of the linear recurrence sequence $\{C_L^{(\ell)}\}_{\ell=1}^\infty$ stated in Corollary~\ref{cor:CEllLinearSeq} is sharp, as shown by the following example.
\begin{example}
Set $q=7,m=2$ and $r=3$. Simply choose $L=Z^{7^3}+Z^{7^2}+Z^7+Z\in\vF_{7^2}[Z]$. Then the linear recurrence sequence $\{C_L^{(\ell)}\}_{\ell\in\N}$ has order $2^r=8$ with minimal polynomial:

$X^8 + 6T^2X^7 + (2T^3 + T^2 + T + 3)X^6 + (5T^3 + T^2)X^5 + (T^6 + T^5 + 6T^4 + 6T^3 + 5T^2 + 4T + 6)X^4 + (5T^5 + 5T^4 + 3T^3 + T^2)X^3 + (2T^7 + 2T^5 + 4T^4 + 6T^3 + T^2 + 3T + 3)X^2 + (6T^8 + 6T^7 + 6T^6 + 6T^5 + 6T^4 + 6T^3 + 6T^2)X + T^8 + 6T^7 + 6T + 1$.
\end{example}

\begin{proposition}\label{prop:CoefficientBound}
Set $k'=\vF_{q^{m\ell}}$ and let $C_L^{(\ell)}$ be the characteristic polynomial of $\phi^{(\ell)}$.
    There exists a characteristic polynomial of degree at most $2^r$ for the sequence $\{C_L^{(\ell)}\}_{\ell\in\N_{>0}}\subseteq \vF_q[T]$ with coefficients having $T$-degrees upper bounded by $m2^{r-1}$. In particular, such bound does not depend on $q$ or $\ell$.
\end{proposition}
\begin{proof}
Using Corollary \ref{cor:CEllLinearSeq}, it simply remains to prove that the coefficients have $T$-degrees bounded by $m2^{r-1}$.
First of all, notice that if the claim is true for a characteristic polynomial for the sequence $\{P_{\phi^{(\ell)}}(1)\}_{\ell\in\N_{>0}}\subseteq\vF_q[T]$ then, because of Equation~\eqref{eq:SameCoeff}, it must be true also for a characteristic polynomial for $\{C_L^{(\ell)}\}_{\ell\in\N_{>0}}\subseteq\vF_q[T]$. Therefore, it is sufficient to prove the statement for the sequence $\{P_{\phi^{(\ell)}}(1)\}_{\ell\in\N_{>0}}\subseteq\vF_q[T]$.

Using the same notation as in the proof of Proposition~\ref{prop:PEll1LinearSeq}, consider the characteristic polynomials $f_{s_1}(X)$ and $f_{s_i}(X)$ for $i\in \{2,\ldots,r\}$ associated to the recurrence relations in Equation~\eqref{eq:S_1Recurrence} and Equation~\eqref{eq:S_iRecurrence} respectively. By Proposition~\ref{prop:LinearAddition}, the polynomial $f(X)=(X-1)\cdot\prod_{i=1}^r f_{s_i}(X)$ is a characteristic polynomial for $\{P_{\phi^{(\ell)}}(1)\}_{\ell\in\N_{>0}}$.

By Riemann Hypothesis for Drinfeld modules (see \cite[Theorem 4.2.7]{Papikian2023Drinfeld}), the coefficient in $f_1(X)$ with the highest degree is $s_r(\beta_1,\ldots,\beta_r)$, whose $T$-degree is at most $m$.

On the other hand, for each $i=2,\ldots,r$, to understand which is the coefficient in $f_{s_i}(X)$ with the highest $T$-degree, we need to know the degree of each $\overline{s_j}(\ldots,\beta_V,\ldots)$ for $j=1,\ldots,\binom{r}{i}$.

Using the same notation as in \cite[Theorem 4.2.7]{Papikian2023Drinfeld}, let $\infty$ be the place at infinity of $\vF_q(T)$, $F_\infty$ be the completion at $\infty$ of $\vF_q(T)$, and $\mathbb{C}_\infty$ be the completion of an algebraic closure of $F_\infty$. Let $|\cdot|_{\mathbb{C}_\infty}$ be the unique extension of the normalized absolute value $q^{-\ord_\infty(\cdot)}$ on $F_\infty$ to $\mathbb{C}_\infty$. Then all the $\beta_i$'s have the same $\infty$-adic absolute value $|\cdot|_{\mathbb{C}_\infty}$ in $\mathbb{C}_\infty$, and notice that $\pi_1=\tau^m$ is one of such $\beta_i$'s.

If $\Tilde{\infty}$ is the unique place of $\Tilde{F}=\vF_q(T,\pi_1)$ above $\infty$, and $e$ is its ramification index, then, from \cite[Equation 4.1.2]{Papikian2023Drinfeld}, we have
\[
\frac{m}{r}=-\frac{1}{e}\ord_{\Tilde{\infty}}(\pi_1).
\]
Notice that $|\cdot|_{\mathbb{C}_\infty}$ is also the unique extension of the absolute value $q^{-\frac{1}{e}\ord_{\Tilde{\infty}}(\cdot)}$ on $\Tilde{F}_{\Tilde{\infty}}$ to $\mathbb{C}_\infty$, where $\Tilde{F}_{\Tilde{\infty}}$ is the completion of $\Tilde{F}$ at $\Tilde{\infty}$.

Now, since $\overline{s_j}$, up to sign, is the $j$-th elementary symmetric polynomial in $\binom{r}{i}$ variables, and $\deg_T(\overline{s_j}(\ldots,\beta_V,\ldots)) = -\ord_\infty(\overline{s_j}(\ldots,\beta_V,\ldots))$, then from ultrametric triangle inequality we get

\begin{align*}
q^{ -\ord_\infty(\overline{s_j}(\ldots,\beta_V,\ldots))} &= \mid \overline{s_j}(\ldots,\beta_V,\ldots) \mid_{\mathbb{C}_\infty} \\
&\leq \max_{\substack{b_1,\dots,b_j,\\ \text{distinct}}} \mid \beta_{V_{b_1}}\cdots\beta_{V_{b_j}} \mid_{\mathbb{C}_\infty} \\
&=  \max_{\substack{b_1,\dots,b_j,\\ \text{distinct}}}  \prod_{k=1}^j\mid \beta_{V_{b_k}} \mid_{\mathbb{C}_\infty} \\
& =  \max_{\substack{b_1,\dots,b_j,\\ \text{distinct}}} \prod_{k=1}^j \prod_{v\in V_{b_k}} \mid \beta_v \mid_{\mathbb{C}_\infty} \\
& = \max_{\substack{b_1,\dots,b_j,\\ \text{distinct}}} \prod_{k=1}^j \prod_{v\in V_{b_k}} \mid \pi_1 \mid_{\mathbb{C}_\infty} = \mid \pi_1 \mid_{\mathbb{C}_\infty}^{ij} \\
& = q^{-\frac{1}{e}\ord_{\Tilde{\infty}}(\pi_1)ij } = q^{ \frac{m}{r}ij }.
\end{align*}
This means that $\deg_T(\overline{s_j}(\ldots,\beta_V,\ldots)) \leq \frac{m}{r}ij $ for $1\leq j\leq\binom{r}{i}$. Therefore, for every $i=2,\ldots,r$, the coefficients in $f_{s_i}(X)$ have $T$-degree at most $\frac{m}{r} i \binom{r}{i}$.

As a result, the coefficients in $f(X)$ (given as the product of the coefficients in the $f_{s_i}$'s) have $T$-degree upper bounded by
\[
m+\sum_{i=2}^r \frac{m}{r} i \binom{r}{i} = \frac{m}{r} \sum_{i=1}^r i \binom{r}{i} = \frac{m}{r}2^{r-1} r = m2^{r-1} .
\]

\end{proof}

\section{Computing $C_L^{(\ell)}$ in $O(n\log^4(n))$}

We conclude by providing our algorithm for the computation of the characteristic polynomial of a linearized polynomial of low $q$-degree, as a linear map over a large extension field.

We already know that we have a linear recurrence relation for $\{C_L^{(\ell)}\}_{\ell\in\N_{>0}}$ of order $d \le 2^r$ as 
\begin{equation}\label{eq:LinearEq}
  C_L^{(k+d)} = c_{d-1} C_L^{(k+d-1)} + \ldots + c_1 C_L^{(k+1)} + c_0 C_L^{(k)},  
\end{equation}
where $c_i\in \vF_q[T]$. Suppose also that for $i=0,\ldots,d-1$ we have $\deg c_i\leq B$, where $B$ does not depend on $q$ or $\ell$. Notice that, by Proposition~\ref{prop:PEll1LinearSeq} and Proposition~\ref{prop:CoefficientBound}, it is always possible to find at least one recurrence relation that satisfies all the conditions (with $d=2^r$ and $B=\frac{1}{2}md$).

Let now $\ell>2^{r+1}$ and $n=m\ell$. We describe now an algorithm to compute $C_L^{(\ell)}$ using the above linear recurrence relation.
\begin{itemize}
    \item[\textbf{Step 1.}] \textbf{Initialization of the algorithm (independent of $\ell$).} For $i=1,\ldots,2^{r+1}$, compute $C_L^{(i)}$. Using standard algorithms already known in literature (see for example \cite{keller1985fast}, \cite{pernet2007faster},\cite{neiger2021deterministic}), each characteristic polynomial can be computed in at most $O((mi)^\omega\log(mi))$ operations over $\vF_q$, where it is $\omega\in(2,3]$. This means that such initialization totally cost at most $O(m^\omega(2^{r+1})^{\omega+1}\log(2md))$ operations over $\vF_q$.
    
    \item[\textbf{Step 2.}] \textbf{Compute the coefficients of the recursion (independent of $\ell$).} Compute coefficients $c_0,\ldots,c_{d-1}\in\vF_q[T]$ of the characteristic polynomial of the linear recurrence relation \eqref{eq:LinearEq}. It is sufficient to solve the following system of linear equations:
    \[\{C_L^{(i+d)} = c_{d-1} C_L^{(i+d-1)} + \ldots + c_1 C_L^{(i+1)} + c_0 C_L^{(i)}\}_{i=1,\ldots,d},\]
    for $c_0,\ldots,c_{d-1}$. This system can be solved using Gaussian elimination with a computational complexity of $O(d^3)$ operations over $\vF_q[T]$ (see for example \cite{nakos1997fraction-free}). 

    \item[\textbf{Step 3.}] \textbf{Compute linear recursion for evaluations of the sequence} In this step we compute evaluation points and their values to recover $C_L^{(\ell)}$ by interpolation at the next step.

    Case 1. Suppose $q>n$ so that we have enough space to choose $n$ evaluation points $x_1,\ldots,x_n$ in $\vF_q$. In the notation of Remark~\ref{rem:LinearRecurNthTerm}, we need to compute the evaluations $U_0|_{x_i}$ and $M^{n+1}|_{x_i}$ for every $i\in \{1,\ldots,n\}$. The computation of $U_0|_{x_i}$ requires at most a number of multiplication over $\vF_q$ equal to:
    \[
    \sum_{i=1}^{d}\deg C_L^{(i)} = \sum_{i=1}^{d} mi = \frac{1}{2}md(d+1).
    \]
    Instead, to compute the evaluation $M^{n+1}|_{x_i}$, it is sufficient to calculate $(M|_{x_i})^{n+1}$ because evaluation is a ring homomorphism. To compute $M|_{x_i}$ we need at most a number of multiplications over $\vF_q$ equal to:
    \[
    %\sum_{i=0}^{d-1}\deg c_i \leq \sum_{i=0}^{d-1}\frac{1}{2}md = \frac{1}{2}md^2.
    \sum_{i=0}^{d-1}\deg c_i \leq \sum_{i=0}^{d-1} B = Bd.
    \]
    Therefore, again by Remark~\ref{rem:LinearRecurNthTerm}, $(M|_{x_i})^{n+1}$ can be computed with
    \[
    %\frac{1}{2}md^2 + d^3\log(n+1)
    Bd + d^3\log(n+1)
    \]
    multiplications over $\vF_q$. Moreover, $(M|_{x_i})^{n+1}\cdot U_0|_{x_i}$ costs $d^2$ multiplications over $\vF_q$. This means that, since we need to repeat the same procedure $n$ times, to compute all the evaluations costs at most $O((md^2+Bd)n + d^3 n\log(n))$ multiplications over $\vF_q$.

    Case 2. Suppose $q<n$. Then, in order to have enough evaluation points, we need to enlarge the base field. It is sufficient to choose points over $\vF_{q^{\lceil \log_q(n)\rceil}}$. Repeating the same argument of Case 1, since the multiplication of two elements over $\vF_{q^{\lceil \log_q(n)\rceil}}$ costs $\log_q^2(n)\approx\log^2(n)$ multiplications over $\vF_q$,  we have that to compute all the evaluations costs at most $O((md^2+Bd)n\log^2(n) + d^3 n\log^3(n))$ multiplications over $\vF_q$.

    \item[\textbf{Step 4.}]\textbf{Interpolation.} Construct $C_L^{(\ell)}$ by interpolation. 
    
    Case 1. Suppose $q>n$. Then interpolation can be done using $O(n\log^2(n))$ arithmetic operations over $\vF_q$ (see for example \cite[Theorem 8.14]{Aho1974Algorithm}).
    
    Case 2. If $q<n$, then interpolation must be performed over an extension field of $\vF_q$. The smallest such extension is $\vF_{q^{\lceil\log_q n\rceil}}$, where each multiplication corresponds to $\log_q^2 n\approx \log^2(n)$ multiplications in $\vF_q$. This means that such interpolation needs $O(n\log^4(n))$ operations over $\vF_q$.
\end{itemize}

Notice that our algorithm works well when $n\gg 0$ (namely $\ell\gg 0$) and $r$ is relatively small. Indeed, in this case, the final cost of the algorithm is dictated by Step 3 and 4, and it is asymptotically equivalent to $O(n\log^3(n))$ when $q>n$ and $O(n\log^4(n))$ when $q<n$. Using standard algorithms known in literature (see  \cite{duan2023faster, keller1985fast, neiger2021deterministic, pernet2007faster}), the same computation requires at least $O(M(n))$ operations to compute the characteristic polynomial directly, where $M(n)$ denotes the complexity of multiplying two $n\times n$ matrices. The best known bound for $M(n)$ is $O(n^{\omega})$ with $\omega\approx 2.371866$.

\section{Acknowledgements}
This work was supported by the National Science Foundation under Grant No 2338424.

\nocite{*} %To show also references not cited
\bibliographystyle{abbrv}
\bibliography{ref.bib}

\begin{thebibliography}{10}

\bibitem{Aho1974Algorithm}
A.~V. Aho and J.~E. Hopcroft.
\newblock {\em The Design and Analysis of Computer Algorithms}.
\newblock Addison-Wesley Longman Publishing Co., Inc., USA, 1st edition, 1974.

\bibitem{bostan2021simple}
A.~Bostan and R.~Mori.
\newblock A simple and fast algorithm for computing the n-th term of a linearly
  recurrent sequence.
\newblock In {\em Symposium on Simplicity in Algorithms (SOSA)}, pages
  118--132. SIAM, 2021.

\bibitem{caruso2025algorithms}
X.~Caruso and A.~Leudi{\`e}re.
\newblock Algorithms for computing norms and characteristic polynomials on
  general drinfeld modules.
\newblock {\em Mathematics of Computation}, 2025.

\bibitem{duan2023faster}
R.~Duan, H.~Wu, and R.~Zhou.
\newblock Faster matrix multiplication via asymmetric hashing.
\newblock In {\em 2023 IEEE 64th Annual Symposium on Foundations of Computer
  Science (FOCS)}, pages 2129--2138, 2023.

\bibitem{dummit2004abstract}
D.~S. Dummit, R.~M. Foote, et~al.
\newblock {\em Abstract algebra}, volume~3.
\newblock Wiley Hoboken, 2004.

\bibitem{Edwards1984Galois}
H.~Edwards.
\newblock {\em Galois Theory}.
\newblock Graduate texts in mathematics. Springer, 1984.

\bibitem{Everest2003Recurrence}
G.~Everest, A.~J. Van Der~Poorten, I.~Shparlinski, T.~Ward, et~al.
\newblock {\em Recurrence sequences}, volume 104.
\newblock American Mathematical Society Providence, RI, 2003.

\bibitem{Hungerford2003Algebra}
T.~Hungerford.
\newblock {\em Algebra}.
\newblock Springer, eighth edition, 2003.

\bibitem{kaltofen2005complexity}
E.~Kaltofen and G.~Villard.
\newblock On the complexity of computing determinants.
\newblock {\em Computational complexity}, 13(3):91--130, 2005.

\bibitem{keller1985fast}
W.~Keller-Gehrig.
\newblock Fast algorithms for the characteristics polynomial.
\newblock {\em Theoretical computer science}, 36:309--317, 1985.

\bibitem{kurakin1995linear}
V.~Kurakin, A.~Kuzmin, A.~Mikhalev, and A.~Nechaev.
\newblock Linear recurring sequences over rings and modules.
\newblock {\em Journal of mathematical sciences}, 76(6):2793--2915, 1995.

\bibitem{kurakin2000polylinear}
V.~Kurakin, A.~Mikhalev, and A.~Nechaev.
\newblock Polylinear recurring sequences over a bimodule.
\newblock In {\em Formal Power Series and Algebraic Combinatorics: 12 th
  International Conference, FPSAC’00, Moscow, Russia, June 2000,
  Proceedings}, pages 484--495. Springer, 2000.

\bibitem{Macdonald1998Symmetric}
I.~Macdonald.
\newblock {\em Symmetric Functions and Hall Polynomials}.
\newblock Oxford classic texts in the physical sciences. Clarendon Press, 1998.

\bibitem{miller1966algorithm}
J.~Miller and D.~S. Brown.
\newblock An algorithm for evaluation of remote terms in a linear recurrence
  sequence.
\newblock {\em The Computer Journal}, 9(2):188--190, 1966.

\bibitem{musleh2023computing}
Y.~Musleh and {\'E}.~Schost.
\newblock Computing the characteristic polynomial of endomorphisms of a finite
  drinfeld module using crystalline cohomology.
\newblock In {\em Proceedings of the 2023 International Symposium on Symbolic
  and Algebraic Computation}, pages 461--469, 2023.

\bibitem{nakos1997fraction-free}
G.~C. Nakos, P.~R. Turner, and R.~M. Williams.
\newblock Fraction-free algorithms for linear and polynomial equations.
\newblock 31(3):11–19, Sept. 1997.

\bibitem{neiger2021deterministic}
V.~Neiger and C.~Pernet.
\newblock Deterministic computation of the characteristic polynomial in the
  time of matrix multiplication.
\newblock {\em Journal of Complexity}, 67:101572, 2021.

\bibitem{newton1967mathematical}
S.~I. Newton.
\newblock {\em The Mathematical Works of Isaac Newton}.
\newblock Johnson Repront Corporation, 1967.

\bibitem{Papikian2023Drinfeld}
M.~Papikian.
\newblock {\em Drinfeld Modules}.
\newblock Springer International Publishing, 2023.

\bibitem{pernet2007faster}
C.~Pernet and A.~Storjohann.
\newblock Faster algorithms for the characteristic polynomial.
\newblock In {\em Proceedings of the 2007 international symposium on Symbolic
  and algebraic computation}, pages 307--314. Association for Computing
  Machinery, 2007.

\bibitem{strassen1969gaussian}
V.~Strassen.
\newblock Gaussian elimination is not optimal.
\newblock {\em Numerische mathematik}, 13(4):354--356, 1969.

\end{thebibliography}

\end{document}